\title{\bf Discrete Quantitative Nodal Theorem}
\author{L\'aszl\'o Lov\'asz\\
Alfr\'ed R\'enyi Institute of Mathematics\footnote{Research supported by ERC
Synergy Grant No.~810115.}}

\documentclass{article}
\usepackage{amssymb,amsmath,graphicx,theorem,bbm}
\usepackage[hypertex]{hyperref}

\nonstopmode

\newtheorem{theorem}{Theorem}

\newtheorem{corollary}[theorem]{Corollary}

\theorembodyfont{\rmfamily}
\newtheorem{example}{Example}

\newenvironment{proof}{\medskip\noindent{\bf Proof. }}{\hfill$\square$\medskip}
\newenvironment{proof*}[1]{\medskip\noindent{\bf Proof of #1.}}{\hfill$\square$\medskip}

\long\def\ignore#1{}

\begin{document}

\def\Pr{{\sf P}}
\def\E{{\sf E}}
\def\T{^\top}
\def\Var{{\sf Var}}
\def\eps{\varepsilon}
\def\wt{\widetilde}
\def\wh{\widehat}
\def\tr{\text{\rm tr}}
\def\supp{\text{\rm supp}}
\def\diag{\text{\rm diag}}

\def\AA{\mathcal{A}}\def\BB{\mathcal{B}}\def\CC{\mathcal{C}}
\def\DD{\mathcal{D}}\def\EE{\mathcal{E}}\def\FF{\mathcal{F}}
\def\GG{\mathcal{G}}\def\HH{\mathcal{H}}\def\II{\mathcal{I}}
\def\JJ{\mathcal{J}}\def\KK{\mathcal{K}}\def\LL{\mathcal{L}}
\def\MM{\mathcal{M}}\def\NN{\mathcal{N}}\def\OO{\mathcal{O}}
\def\PP{\mathcal{P}}\def\QQ{\mathcal{Q}}\def\RR{\mathcal{R}}
\def\SS{\mathcal{S}}\def\TT{\mathcal{T}}\def\UU{\mathcal{U}}
\def\VV{\mathcal{V}}\def\WW{\mathcal{W}}\def\XX{\mathcal{X}}
\def\YY{\mathcal{Y}}\def\ZZ{\mathcal{Z}}

\def\Ab{\mathbf{A}}\def\Bb{\mathbf{B}}\def\Cb{\mathbf{C}}
\def\Db{\mathbf{D}}\def\Eb{\mathbf{E}}\def\Fb{\mathbf{F}}
\def\Gb{\mathbf{G}}\def\Hb{\mathbf{H}}\def\Ib{\mathbf{I}}
\def\Jb{\mathbf{J}}\def\Kb{\mathbf{K}}\def\Lb{\mathbf{L}}
\def\Mb{\mathbf{M}}\def\Nb{\mathbf{N}}\def\Ob{\mathbf{O}}
\def\Pb{\mathbf{P}}\def\Qb{\mathbf{Q}}\def\Rb{\mathbf{R}}
\def\Sb{\mathbf{S}}\def\Tb{\mathbf{T}}\def\Ub{\mathbf{U}}
\def\Vb{\mathbf{V}}\def\Wb{\mathbf{W}}\def\Xb{\mathbf{X}}
\def\Yb{\mathbf{Y}}\def\Zb{\mathbf{Z}}

\def\ab{\mathbf{a}}\def\bb{\mathbf{b}}\def\cb{\mathbf{c}}
\def\db{\mathbf{d}}\def\eb{\mathbf{e}}\def\fb{\mathbf{f}}
\def\gb{\mathbf{g}}\def\hb{\mathbf{h}}\def\ib{\mathbf{i}}
\def\jb{\mathbf{j}}\def\kb{\mathbf{k}}\def\lb{\mathbf{l}}
\def\mb{\mathbf{m}}\def\nb{\mathbf{n}}\def\ob{\mathbf{o}}
\def\pb{\mathbf{p}}\def\qb{\mathbf{q}}\def\rb{\mathbf{r}}
\def\sb{\mathbf{s}}\def\tb{\mathbf{t}}\def\ub{\mathbf{u}}
\def\vb{\mathbf{v}}\def\wb{\mathbf{w}}\def\xb{\mathbf{x}}
\def\yb{\mathbf{y}}\def\zb{\mathbf{z}}

\def\Abb{\mathbb{A}}\def\Bbb{\mathbb{B}}\def\Cbb{\mathbb{C}}
\def\Dbb{\mathbb{D}}\def\Ebb{\mathbb{E}}\def\Fbb{\mathbb{F}}
\def\Gbb{\mathbb{G}}\def\Hbb{\mathbb{H}}\def\Ibb{\mathbb{I}}
\def\Jbb{\mathbb{J}}\def\Kbb{\mathbb{K}}\def\Lbb{\mathbb{L}}
\def\Mbb{\mathbb{M}}\def\Nbb{\mathbb{N}}\def\Obb{\mathbb{O}}
\def\Pbb{\mathbb{P}}\def\Qbb{\mathbb{Q}}\def\Rbb{\mathbb{R}}
\def\Sbb{\mathbb{S}}\def\Tbb{\mathbb{T}}\def\Ubb{\mathbb{U}}
\def\Vbb{\mathbb{V}}\def\Wbb{\mathbb{W}}\def\Xbb{\mathbb{X}}
\def\Ybb{\mathbb{Y}}\def\Zbb{\mathbb{Z}}

\def\R{{\mathbb R}}
\def\Q{{\mathbb Q}}
\def\Z{{\mathbb Z}}
\def\N{{\mathbb N}}
\def\C{{\mathbb C}}
\def\U{{\mathbb U}}
\def\Ge{{\mathbb G}}
\def\Ha{{\mathbb H}}

\def\lunl{[\hskip-1pt[}
\def\runl{]\hskip-1pt]}
\def\one{{\mathbbm1}}

\def\sbd#1{{#1}^\text{\rm sub}}

\maketitle

\begin{abstract}
We prove a theorem that can be thought of as a common generalization of the
Discrete Nodal Theorem and (one direction of) Cheeger's Inequality for graphs.
A special case of this result will assert that if the second and third
eigenvalues of the Laplacian are at least $\eps$ apart, then the subgraphs
induced by the positive and negative supports of the eigenvector belonging to
$\lambda_2$ are not only connected, but edge-expanders (in a weighted sense,
with expansion depending on $\eps$).
\end{abstract}

\section{Introduction}

In the theory of Riemannian manifolds, two basic theorems connect the geometry
of the manifold to the spectrum of the Laplace operator on the manifold:
Courant's Nodal Theorem and Cheeger's Inequality. Both of these have analogues
in graph theory.

It is a basic simple fact that the graph is connected if and only if the
smallest eigenvalue of its (combinatorial) Laplacian (which is always $0$) has
multiplicity one. Discrete Cheeger inequalities (Alon and Milman \cite{AlMi},
Alon \cite{Al}, Dodziuk and Kendall \cite{DoKe}, Jerrum and Sinclair
\cite{JeSi}) give a quantitative version of this: Roughly speaking, a graph is
an expander if and only if the second smallest eigenvalue of its Laplacian is
bounded away from zero.

The simplest version of the Discrete Nodal Theorem asserts that if $x$ is an
eigenvector of the Laplacian of a connected graph $G$ belonging to the second
smallest eigenvalue, and this eigenvalue has multiplicity one, then the
positive and negative supports of $x$ induce connected subgraphs (Fiedler
\cite{Fied}). If the second smallest eigenvalue has higher multiplicity, there
are exceptions (a simple example is the 3-star), but they can be characterized
\cite{HvH,HLS}.

The Discrete Nodal Theorem was extended to higher eigenvectors by Fiedler
\cite{Fied}, Colin de Verdi\`ere \cite{CdV93}, Davies, Gladwell, Leytold and
Stadler \cite{DavGL}, and Duval and Reiner \cite{DuRe}; to embedded graphs by
Lin, Lippner, Mangoubi and Yau \cite{LLMY}; see also B\i y\i ko$\breve{\rm
g}$lu, Leydold and Stadler \cite{BLS}. To motivate our results, let us quote a
simple version. A {\it nodal domain} of a vector $v\in\R^V$ is a connected
component of the subgraph induced by its positive support, or a connected
component of the subgraph induced by its negative support.

\begin{theorem}
Let $G$ be a connected graph, let $\lambda_1=0\le\lambda_2\le \lambda_3\le
\dots\le\lambda_n$ be the eigenvalues of its Laplacian, and assume that
$\lambda_k<\lambda_{k+1}$ for some $k\ge2$. Let $y$ be an eigenvector belonging
to $\lambda_k$. Then the number of nodal domains of $y$ is at most $k$.
\end{theorem}

There are conditions other than $\lambda_k<\lambda_{k+1}$ to guarantee that the
number of nodal domains is at most $k$, for example, that $y$ has minimal
support among all eigenvectors belonging to $\lambda_k$.

In this note we prove a theorem that can be thought of as a common
generalization of the Discrete Nodal Theorem and (one direction of) Cheeger's
Inequality for graphs. A special case of our results will assert that if the
second and third eigenvalues of the Laplacian are at least $\eps$ apart, then
the subgraphs induced by the positive and negative supports of the eigenvector
belonging to $\lambda_2$ are not only connected, but edge-expanders (in a
weighted sense, with an expansion rate depending on $\eps$).

\section{The main result}

Let $(w_i:~i\in V)$ be a weighting of the nodes of the graph $G=(V,E)$ with
nonnegative weights. For $S\subseteq V$, set $w(S) = \sum_{i\in S} w_i$, and
$\nabla(S)=\{ij\in E:~i\in S,j\in V\setminus S\}$. If $w(S)>0$, we define the
{\it edge-expansion} of $S$ as
\[
\Psi_w(S) = \Psi_{G,w}(S) = \frac1{w(S)}\sum_{ij\in\nabla(S)} \sqrt{w_iw_j}\,.
\]
We say that $G$ is a {\it $c$-expander with respect to $w$} $(c>0)$, if
$\Psi_w(S)\ge c$ for every subset $S\subseteq V$ with $0<w(S)<w(V)/2$.

To generalize this notion to multiway cuts, it is easier to formulate the
contrapositive. We say that $(G,w)$ is {\it $(k,c)$-partitionable}, if $V$ has
a partition $V=S_1\cup\dots\cup S_k$ into sets with $w(S_i)>0$ such that
$\Psi_w(S_i) < c$ for all $1\le i\le k$. It is easy to check that if $k>2$,
then merging two classes in such a partition, the new class $S_i\cup S_j$
satisfies $\Psi_w(S_i\cup S_j) < c$. Hence every $(k,c)$-partitionable weighted
graph is $(k-1,c)$-partitionable as well.

For a vector $x\in\R^V$, we denote its positive and negative support by
$\supp_+(x)=\{i\in V:~x_i>0\}$ and $\supp_-(x)=\{i\in V:~x_i<0\}$, and by
$G_x^+$ and $G_x^-$, the subgraphs of $G$ induced by $\supp_+(x)$ and
$\supp_-(x)$ respectively.

\begin{theorem}\label{THM:1}
Let $\lambda_1=0\le\lambda_2\le \lambda_3\le \dots\le\lambda_n$ be the
eigenvalues of the Laplacian $L$ of a graph $G=(V,E)$. Let $y$ be an
eigenvector belonging to $\lambda_k$ $(1\le k\le n)$, and set $w_i=y_i^2$ and
$c=(\lambda_{k+1}-\lambda_k)/2$. Suppose that the weighted graph $(G_y^+, w)$
is $(a,c)$-partitionable, and $(G_y^-, w)$ is $(b,c)$-partitionable. Then
$a+b\le k$.
\end{theorem}

\begin{proof}
Let us write $G^+=G^+_y$, $V^+=\supp_+(y)$ and $\Psi^+(S)=\Psi_{G^+,w}(S)$ for
$S\subseteq V^+$, and define $G^-$, $V^-$ and $\Psi^-$ analogously. Note that
$\Psi^+(S)=0$ is $S$ induces a connected component of $G^+$. Let
$V^+=V_1\cup\dots\cup V_a$ be a partition with $\Psi^+(V_i) < c$, and let
$V^-=V_{a+1}\cup\dots\cup V_{a+b}$ be an analogous partition. Let us assume (by
way of contradiction) that $a+b\ge k+1$; we may assume (by merging partition
classes) that $a+b=k+1$.

Let $M=L-\lambda_kI$, so that $My=0$. Let $y^i\in \R^V$ denote the vector
obtained from $y$ by replacing all entries in $V\setminus V_i$ by $0$. This
vector is nonzero, because $|y^i|^2 = w(V_i)>0$ by the definition of
$(k,c)$-partitionable graphs. Furthermore, $y^i\ge0$ for $1\le i\le a$,
$y^i\le0$ for $a+1\le i\le a+b$, and $y=y^1+\dots+y^{k+1}$. Let $z_i=|y^i|$,
$\wh{y}^i=(1/z_i)y^i$, and $z=(z_1,\dots,z_{k+1})\T$. Consider the
$(k+1)\times(k+1)$ matrix $B$ defined by
\[
B_{ij}=\langle \wh{y}^i, M\wh{y}^j\rangle = \frac1{z_iz_j}\langle y^i, My^j\rangle,
\]
and let $\mu_1\le\dots\le\mu_{k+1}$ be its eigenvalues. Note that for $i\le a$,
we have
\begin{equation}\label{EQ:PSI-PLUS}
\Psi^+(V_i) = \frac1{z_i^2}\sum_{u\in V_i}\sum_{v\in V_+\setminus V_i} y_uy_v
= -\frac1{z_i^2}\sum_{j\le a, j\not=i} \langle y^i, My^j\rangle =-\frac{1}{z_i}\sum_{j\le a, j\not=i}B_{ij}z_j.
\end{equation}
Analogous formula holds for $i>a$.

Let us start with some elementary properties of $B$. We have $\langle y^i, My^j
\rangle\le 0$ if $i\not=j$ and $1\le i,j\le a$, which implies that $B_{ij}\le0$
in this case. Similarly $B_{ij}\le0$ if $i\not=j$ and $a+1\le i,j\le k+1$, and
$B_{ij}\ge0$ if $i\le a<j$, or the other way around. Furthermore, we have
$Bz=0$; indeed,
\[
(Bz)_i = \sum_{j=1}^{k+1} \frac1{z_i}\langle y^i, My^j\rangle =
\frac1{z_i}\langle y^i, My\rangle =0.
\]

Since the vectors $\wh{y}^i$ ($i=1,\dots, k+1$) are mutually orthogonal unit
vectors, the matrix $B$ is a principal submatrix of $M$ in an appropriate
orthonormal basis. By the Interlacing Eigenvalues Theorem, we have
\begin{equation}\label{EQ:X}
\lambda_i-\lambda_k\le \mu_i \qquad(i=1,\dots, k+1).
\end{equation}

Let $C$ be a symmetric $(k+1)\times(k+1)$ matrix, given by
\[
C_{ij}=
  \begin{cases}
   B_{ij}  & \text{if $i\not=j$ and either $i,j\le a$ or $i,j\ge a+1$}, \\
   \Psi^+(V_i), & \text{if $i=j\le a$}, \\
   \Psi^-(V_i), & \text{if $i=j\ge a+1$}, \\
   0,  & \text{otherwise}.
  \end{cases}
\]
Then for $1\le i\le a$, using \eqref{EQ:PSI-PLUS},
\begin{align*}
(Cz)_i &= \Psi^+(V_i) z_i + \sum_{j\le a, j\not=i} B_{ij} z_j = 0.
\end{align*}
Similar computation works for $i>a$, to get $Cz=0$.

\medskip

\noindent{\bf Claim.} {\it The matrices $C$ and $C-B$ are positive
semidefinite.}

\medskip

Indeed, from our discussion of the signs of the entries of $B$, it follows that
the off-diagonal entries of $C$ and of $C-B$ are nonpositive. We also have
$(C-B)z=Cz=Bz=0$, which implies that their diagonal entries are nonnegative.
Let $D=\diag(z)$, then the matrix $DCD$ has nonnegative entries in the
diagonal, nonpositive entries everywhere else, and every row-sum is $0$. So
this matrix is diagonally dominant, and hence positive semidefinite, which
implies that $C$ is positive semidefinite. For the matrix $C-B$ the conclusion
follows similarly.

Next, we show that the largest eigenvalue of the matrix $C$ satisfies
\begin{equation}\label{EQ:X1}
\lambda_{\max}(C)<2c.
\end{equation}
Indeed, let $u$ be the eigenvector of $C$ belonging to $\lambda_{\max}(C)$. We
may assume that $u_1=z_1>0$ and $|u_i|\le z_i$ for all $i$. Then, using that
$C_{1i}\le 0$ for $i\not=1$, we get
\[
\lambda_{\max}(C) u_1 = \sum_i C_{1i} u_i \le
C_{11}u_1 - \sum_{i>1} C_{1i}z_i =2C_{11}u_1<2cu_1.
\]
This proves \eqref{EQ:X1}.

Positive semidefiniteness of $C-B$ implies that $\mu_{k+1}$, the largest
eigenvalue of $B$, is bounded above by $\lambda_{\max}(C)$. Hence by
\eqref{EQ:X},
\begin{equation}\label{EQ:X0}
\lambda_{k+1}-\lambda_k \le \mu_{k+1} \le \lambda_{\max}(C)<2c,
\end{equation}
which contradicts the choice of $c$.
\end{proof}

The case $k=2$ is worth stating separately:

\begin{corollary}\label{COR:1}
If $y$ is an eigenvector belonging to $\lambda_2$, then both $G_y^+$ and
$G_y^-$ are $(\lambda_3-\lambda_2)/2$-expanders with respect to the weights
$y_i^2$.
\end{corollary}

From our considerations, we can derive two other inequalities:

\begin{corollary}\label{COR:2}
If $y$ is an eigenvector belonging to $\lambda_k$, $a+b=k+1$,
$\{U_1,\dots,U_a\}$ is a partition of $\supp_+(y)$ and $\{V_1,\dots,V_b\}$ is a
partition of $\supp_-(y)$ with $w(U_i),w(V_i)>0$, then
\begin{equation*}
\lambda_{k+1}-\lambda_k\le 2\max\big\{\Psi^+(U_1),\dots,\Psi^+(U_a),
\Psi^-(V_1),\dots,\Psi^-(V_b)\big\}.
\end{equation*}
and
\[
\lambda_{k+1}-\lambda_k\le \sum_{i=1}^a \Psi^+(U_i) +
\sum_{i=1}^b \Psi^-(V_i).
\]
\end{corollary}

(The second inequality is stronger in those cases only when the expansions of
the sets are very different)

\begin{proof}
The first inequality is an easy rephrasing of Theorem \ref{THM:1}. To prove the
second, it suffices to notice that the matrix $C$ in the proof above is
positive semidefinite. Hence its largest eigenvalue is bounded above by its
trace. Since $C-B$ is positive semidefinite, it follows that
\[
\mu_{a+b}\le \tr(C)=\sum_i C_{ii} = \sum_{i=1}^a \Psi^+(U_i) +
\sum_{i=1}^b \Psi^-(V_i).
\]
\end{proof}

\section{Examples}

There is no easy converse to Corollary \ref{COR:1}. The following example shows
that even if $\lambda_2=\lambda_3$, no separation property for the positive and
negative supports of any eigenvector belonging to $\lambda_2$ follows.

\begin{example}\label{EXA:CYCLIC}
Let $C_n$ denote the cycle of length $n$, with its nodes labeled
$0,1,\dots,n-1$ in the natural order. Let $C_n^k$ denote the graph obtained
from $C_n$ by connecting any two nodes at distance of at most $k$ along the
cycle. The eigenvalues of the adjacency matrix are $\mu_r= 2\sum_{h=1}^k
\cos(hr\pi i/n)$, $r=0,1,\dots n-1$. The smallest eigenvalue of the Laplacian
is $\lambda_1=2-\mu_0=0$, and it is not hard to see that the second smallest
eigenvalues are $\lambda_2 =\lambda_3=2-\mu_1=2-\mu_{n-1}$. So the gap
$\lambda_3-\lambda_2=0$. For every eigenvector $x$ belonging to $\lambda_2$,
its positive support is induced by a half-cycle, which is easily seen to be a
$c$-expander (even if weighted with the squared entries of the eigenvector) for
$c=\Theta(k^2/n^2)$.
\end{example}

The assertion of the main theorem (or of its corollary) does not remain true
without the weights, as the following example shows.

\begin{example}\label{EXA:2EXPANDERS}
Consider two isomorphic $D$-regular expanders $G_1$ and $G_2$ with $p$ nodes.
Connect two corresponding nodes $a_1$ and $a_2$ in $G_1$ and $G_2$ by a path
$P$ of length $q+1$, to get a connected graph $G$ with $n=2p+q$ nodes. We
assume that $q,p\to\infty$ and $q=o(p)$. Let $\lambda_1=0<
\lambda_2\le\dots\le\lambda_n$ be its eigenvalues of its Laplacian, with unit
length eigenvectors $v_1=(1/\sqrt{n})\one,v_2,\dots,v_n$.

It is not hard (but a little tedious) to see that this is a counterexample.
Informally, the graph is ``almost disconnected'', and hence $\lambda_2$ is
small (less than $2/p$). On the other hand, $\lambda_3-\lambda_2$ will be of
the same order as the eigenvalue gap of $G_1$, which is of constant magnitude.
The positive support of $v_2$ will consist of $G_1$ and half of the path $P$,
which is not an expander in the unweighted sense.
\end{example}

\end{document}